\documentclass[11pt]{article}
\usepackage{amsfonts}
\usepackage{latexsym,amssymb,amsmath,graphics,cite,arydshln}
\usepackage{tikz}
\usepackage{fancyhdr}
\usepackage{color}

\usepackage{multirow} 
\usepackage{xcolor}
\usepackage{amssymb}
\usepackage{geometry}
\usepackage{lipsum}
\usepackage{enumerate}

\newcommand{\ignore}[1]{}

\begin{document}
\newcommand{\qed}{\hphantom{.}\hfill $\Box$\medbreak}
\newcommand{\proof}{\noindent{\bf Proof \ }}

\newtheorem{theorem}{Theorem}[section]
\renewcommand{\theequation}{\thesection.\arabic{equation}}
\newtheorem{lemma}[theorem]{Lemma}
\newtheorem{corollary}[theorem]{Corollary}
\newtheorem{example}[theorem]{Example}
\newtheorem{remark}[theorem]{Remark}
\newtheorem{definition}[theorem]{Definition}
\newtheorem{construction}[theorem]{Construction}
\newtheorem{fact}[theorem]{Fact}
\newtheorem{proposition}[theorem]{Proposition}
\newtheorem{conjecture}[theorem]{Conjecture}

\title{On nontrivial cross-$2$-intersecting families \footnote{Corresponding author: Lijun Ji}
}

\author{{\small Jingjun Bao$^1$, Lijun Ji$^2$, Liying Yu$^2$}\\
{\small $^1$ School of Mathematics and Statistics, Ningbo University, Ningbo 315211, P. R. China}\\
{\small $^2$ Department of Mathematics, Soochow University, Suzhou 215006, P. R. China}\\
{\small E-mail:  \ baojingjun@hotmail.com;  jilijun@suda.edu.cn; lyyu2000@stu.suda.edu.cn}\\
}

%
\date{}
\maketitle

\begin{abstract}
	
Two families \(\mathcal{A}\subseteq\binom{[n]}{k}\) and \(\mathcal{B}\subseteq\binom{[n]}{\ell}\) are  said to be nontrivial cross-\(t\)-intersecting if \(|A \cap B| \geq t\) for all \(A \in \mathcal{A}\) and \(B \in \mathcal{B}\), and $|\bigcap_{A\in \mathcal{A}\cup \mathcal{B}}A|<t$. In this paper, we determine the upper bound on \(|\mathcal{A}||\mathcal{B}|\) of two nontrivial cross-\(2\)-intersecting families \(\mathcal{A}\subseteq\binom{[n]}{k}\) and \(\mathcal{B}\subseteq\binom{[n]}{\ell}\)  for any positive integers $n,k,\ell$ with \(k\geq \ell  \geq 3\) and \(n \geq 3(k-1)\). Moreover, we characterize the extremal families attaining this bound. This settles the last unsolved case of a recent result by He, Li, Wu and Zhang (J. Combin. Theory Ser. A, 217 (2026) 106095).

\smallskip
\end{abstract}

\section{Introduction}
Let \(n\), \(k\) and \(t\) be positive integers with \(t \leq k \leq n\). We denote the set \(\{1, \ldots, n\}\) by \([n]\) and use \(2^{[n]}\) to represent its power set. A nonempty subset \(\mathcal{F}\) of \(2^{[n]}\) is called \(k\)-{\em uniform} if its elements all have size \(k\). For any finite set \(D\), we use \(\binom{D}{k}\) to denote the collection of all \(k\)-element subsets of \(D\). We say that \(\mathcal{F} \subset 2^{[n]}\) is a \(t\)-{\em intersecting family} if \(|A\cap B|\geq t\) for all \(A,B\in\mathcal{F}\). A \(1\)-intersecting family is simply referred to as an intersecting family. The Erd\H{o}s–Ko–Rado Theorem is one of the central results in extremal combinatorics.

\begin{theorem}[Erd\H{o}s–Ko–Rado Theorem \cite{EKR1961}]\label{EKR}
Let \(k\), \(n\) and \(t\) be positive integers such that \(t \leq k \leq n\). If \(\mathcal{F}\subseteq\binom{[n]}{k}\) is \(t\)-intersecting, then there is a constant $n_0(k,t)$ such that for \(n\geq n_{0}(k,t)\), the following inequality holds. \[|\mathcal{F}|\leq\binom{n - t}{k - t}.\]
\end{theorem}

Let \(k, t\) be fixed positive integers with \(t \leq k\). Denote by \(N_{0}(k,t)\) the smallest possible value of \(n_{0}(k,t)\) in Theorem \ref{EKR}. For the case \(t = 1,\) it was shown in \cite{EKR1961} that \(N_{0}(k,1)=2k\). For \(t\geq 1,\) the value of \(N_{0}(k,t)\) is given by \(N_{0}(k,t)=(t + 1)(k - t+1)\), which was established in \cite{F1978} for \(t\geq15\) and in \cite{W1984} for all \(t\).

Let \(r\) be an integer such that \(0\leq r\leq\frac{n - t}{2}\). Define the following \(k\)-uniform family 
\[\mathcal{F}(n,k,t,r):=\left\{A\in\binom{[n]}{k}:|A\cap[t + 2r]|\geq t + r\right\}.\] 
Then it is a \(t\)-intersecting family. Frankl \cite{F1978} conjectured that if \(\mathcal{F}\) is a \(t\)-intersecting subfamily of \(\binom{n}{k}\), then \[|\mathcal{F}|\leq\max_{0\leq r\leq(n - t)/2}|\mathcal{F}(n,k,t,r)|.\]

This conjecture was proved partially by Frankl and Füredi \cite{FF1986}, and then settled completely by Ahlswede and Khachatrian \cite{AK1997}.

\begin{theorem}[Ahlswede and Khachatrian \cite{AK1997}]\label{AK97}
Let \(n, k\) and \(t\) be positive integers such that \(t \leq k \leq n\). Let \(\mathcal{F}\) be a \(t\)-intersecting \(k\)-uniform family of subsets of \([n]\).
\begin{itemize}
\item[(i)] If \((k - t + 1)(2+\frac{t - 1}{r+1})<n<(k - t + 1)(2+\frac{t - 1}{r})\) for some nonnegative integer \(r\), then \(|\mathcal{F}|\leq|\mathcal{F}(n,k,t,r)|\) and the equality holds if and only if \(\mathcal{F}=\mathcal{F}(n,k,t,r)\) up to isomorphism.
\item[(ii)] If \((k - t + 1)(2+\frac{t - 1}{r+1})=n\) for some nonnegative integer \(r\), then \(|\mathcal{F}|\leq|\mathcal{F}(n,k,t,r)| = |\mathcal{F}(n,k,t,r + 1)|\) and the equality holds if and only if \(\mathcal{F}=\mathcal{F}(n,k,t,r)\) or \(\mathcal{F}(n,k,t,r + 1)\) up to isomorphism.
\end{itemize}
\end{theorem}

Let \(\mathcal{A}\) and \(\mathcal{B}\) be two families of subsets of \([n]\). They are called {\em cross-\(t\)-intersecting} if \(|A\cap B|\geq t\) for all \(A\in\mathcal{A}\) and \(B\in\mathcal{B}\). In the special case where \(t = 1,\) we say that \(\mathcal{A}\) and \(\mathcal{B}\) are {\em cross-intersecting}. The cross-\(t\)-intersecting property can be seen as a natural generalization of the \(t\)-intersecting property, as the two properties coincide when \(\mathcal{A}=\mathcal{B}\). 

Now, consider cross-\(t\)-intersecting families \(\mathcal{A}\subseteq \binom{[n]}{k}\) and \(\mathcal{B}\subseteq \binom{[n]}{\ell}\). We say that these families are {\em maximum} if there do not exist cross-\(t\)-intersecting families \(\mathcal{A}_{1}\subseteq \binom{[n]}{k}\) and \(\mathcal{B}_{1}\subseteq \binom{[n]}{\ell}\) such that  \(|\mathcal{A}||\mathcal{B}|<|\mathcal{A}_{1}||\mathcal{B}_{1}|\). Similarly, we say that \(\mathcal{A}\) and \(\mathcal{B}\) are {\em maximal} if there do not exist cross-\(t\)-intersecting families \(\mathcal{A}\subseteq \mathcal{A}_{1} \subset {[n]\choose k}$ and $\mathcal{B}\subseteq \mathcal{B}_{1} \subset {[n]\choose \ell}\) such that  \(|\mathcal{A}||\mathcal{B}|<|\mathcal{A}_{1}||\mathcal{B}_{1}|\). 

The study of possible maximum sizes of pairs of cross-intersecting families is one of the central problems of extremal set theory. In 1986, Pyber \cite{P1986} employed the method of cyclic permutations to extend the Erd\H{o}s–Ko–Rado Theorem to the cross-intersecting case.
For \(k> \ell,\) the lower bound on \(n\) given by Pyber \cite{P1986} turns out not to be sharp. In 1989, Matsumoto and Tokushige \cite{MT1989} derived a sharper result.

\begin{theorem}[Matsumoto and Tokushige \cite{MT1989}]
Let \(n, k\) and  \(\ell\) be positive integers such that \(n\geq 2k \geq 2\ell\). If \(\mathcal{A}\subseteq\binom{[n]}{k}\) and \(\mathcal{B}\subseteq\binom{[n]}{\ell}\) are cross-intersecting, then \(|\mathcal{A}||\mathcal{B}|\leq\binom{n - 1}{k - 1}\binom{n - 1}{\ell - 1}\). Moreover, the equality holds if and only if \(\mathcal{A}=\left\{A\in\binom{[n]}{k}:x\in A\right\}\) and \(\mathcal{B}=\left\{B\in\binom{[n]}{\ell}:x\in B\right\}\) for some fixed element \(x\in[n]\).
\end{theorem}

For the general cross-\(t\)-intersecting case, Tokushige \cite{T2010} employed a combinatorial approach to derive a lower bound on \(n\geq 2tk\) for cross-\(t\)-intersecting families, and further improved this lower bound to \(\frac{k}{n}<1-\frac{1}{\sqrt[t]{2}}\) by using the eigenvalue method in 2013 \cite{T2013}. 
In the same paper \cite{T2013}, Tokushige conjectured that for \(n\geq(t + 1)(k - t + 1)\) with \(k\geq t\), if \(\mathcal{A}\subseteq\binom{[n]}{k}\) and \(\mathcal{B}\subseteq\binom{[n]}{k}\) are cross-\(t\)-intersecting, then \(|\mathcal{A}||\mathcal{B}|\leq\binom{n - t}{k - t}^2\). Moreover, the equality holds if and only if \(\mathcal{A}=\mathcal{B}=\left\{F\in\binom{[n]}{k}:T\subset F\right\}\) for some \(t\)-subset \(T\) of \([n]\).
In 2014, Frankl, Lee, Siggers and Tokushige \cite{FLST2014} verified this conjecture under the stronger assumptions that \(t\geq14\) and \(n\geq(t + 1)k\). In the same year, Borg \cite{B2014} independently confirmed the conjecture for sufficiently large \(n\). More recently, Zhang and Wu \cite{ZW2025} used shifting techniques and generating set analysis to confirm the conjecture for $t\geq 3.$ Also in 2025, Tanaka and Tokushige \cite{TT2025} employed a semidefinite programming approach to resolve the case $t=2.$
 
 \begin{theorem}[Zhang and Wu \cite{ZW2025}, Tanaka and Tokushige \cite{TT2025}]\label{T13}
 Let \(n, k, t\) be positive integers such that \(n \geq k \geq t\) and \(n\geq(t + 1)(k - t + 1)\). If \(\mathcal{A}\subseteq\binom{[n]}{k}\) and \(\mathcal{B}\subseteq\binom{[n]}{k}\) are cross-\(t\)-intersecting, then \(|\mathcal{A}||\mathcal{B}|\leq\binom{n - t}{k - t}^2\). 
 \end{theorem}

In 2013, Tokushige \cite{T2013} proposed a conjecture on a lower bound on \(n\) for non-uniform cross-\(t\)-intersecting families: For positive integers \(k\geq \ell\geq t\geq 1\) and \(n\geq (t+1)(k-t+1)\), if two families \(\mathcal{A}\subseteq\binom{[n]}{k}\) and \(\mathcal{B}\subseteq\binom{[n]}{\ell}\) are cross-\(t\)-intersecting, then \(|\mathcal{A}||\mathcal{B}|\leq\binom{n - t}{k - t}\binom{n - t}{\ell - t}\). In 2016, Borg \cite{B2016} verified the conjecture under the stronger assumptions that \(n\geq (t + 4)(k-t)+\ell-1\). In 2025, He et al. \cite{HLWZ2026} employed shifting techniques and generating set analysis to confirm this conjecture for $t\geq 3.$ More recently, Chen et al. \cite{CLWZ2025} also used the generating set method to verify the conjecture for $t=2$ with $n\geq 3.38k$. Recently, it has been completely solved by the first two authors and Xiang \cite{BJX2026}. 

\begin{theorem}[He et al. \cite{HLWZ2026}, Chen et al. \cite{CLWZ2025}, Bao et al. \cite{BJX2026}]\label{TC13}
Let $k,\ell,t$ and $n$ be positive integers with \(k\geq \ell\geq t\) and \(n\geq (t+1)(k-t+1)\). Suppose that two families \(\mathcal{A}\subseteq\binom{[n]}{k}\) and \(\mathcal{B}\subseteq\binom{[n]}{\ell}\) are cross-\(t\)-intersecting. Then \(|\mathcal{A}||\mathcal{B}|\leq\binom{n - t}{k - t}\binom{n - t}{\ell - t}\). Moreover, the equality holds if and only if, up to isomorphism, one of the
	following holds:
\begin{itemize}
	\item[{\rm (1)}]\(\mathcal{A}=\left\{A \in \binom{[n]}{k} : T \subset A \right\}\) and \(\mathcal{B}=\left\{B \in \binom{[n]}{\ell} : T \subset B \right\}\) for some \(t\)-element subset \(T\);
	\item[{\rm (2)}] \(n=(t+1)(k-t+1),\ k=\ell\) and \(\mathcal{A}=\mathcal{B}=\left\{A \in \binom{[n]}{k} :  |A \cap T| \geq t+1\right\}\) for some \((t+2)\)-element subset \(T\).
\end{itemize}	
	
\end{theorem}

Two cross-$t$-intersecting families \(\mathcal{A}\subseteq\binom{[n]}{k}\) and \(\mathcal{B}\subseteq\binom{[n]}{\ell}\) are  said to be nontrivial if $|\bigcap_{A\in \mathcal{A}\cup \mathcal{B}}A|<t$.  

Frankl and Kupavskii sharpened the Pyber’s inequality when \(\mathcal{A}\subseteq\binom{[n]}{k}\) and \(\mathcal{B}\subseteq\binom{[n]}{\ell}\) are nontrivial cross-intersecting. To be more precise, they obtained the following theorem of the nontrivial cross-intersecting
case.
\begin{theorem}
[Frankl and Kupavskii \cite{FK2017}] Let $n$ and $k$ be two positive integers with $n > 2k$. 
Suppose that  \(\mathcal{A}\) and \(\mathcal{B}\) are cross-intersecting families of \(\binom{[n]}{k}\). If  \(\mathcal{A}\) and \(\mathcal{B}\) are nontrivial, then
\(|\mathcal{A}||\mathcal{B}|\leq\left (\binom{n - 1}{k - 1}+1\right )\left (\binom{n - 1}{k - 1}-\binom{n-k-1}{k-1}\right )\).
\end{theorem}

For $t+1\leq s\leq k+1,$ set
\begin{align*}
&\mathcal{A}_{n,k,s,t}=\{A\in {[n]\choose k}\colon |A\cap[s]|\geq t\}, \\
&\mathcal{B}_{n,\ell,s,t} =\{B\in {[n]\choose \ell}\colon [s]\subset B\}, \\
&\mathcal{H}_{n,k,s,t}=\{A\in {[n]\choose k}\colon \text{either}\ [t]\subset A\ \text{or}\ |A\cap[s]|\geq s-1\},\\ & \mathcal{I}_{n,\ell,s,t}=\{B\in {[n]\choose \ell}\colon [t]\subset B\ \text{and}\ |B\cap[s]|\geq t+1\}.
\end{align*}

It is obvious that $\mathcal{B}_{n,\ell,t+1,t} =\mathcal{I}_{n,\ell,t+1,t}$, $\mathcal{A}_{n,k,t+1,t} = \mathcal{H}_{n,k,t+1,t}$, and $\mathcal{B}_{n,\ell,s,t}$ and $\mathcal{A}_{n,k,s,t}$, $\mathcal{H}_{n,k,s,t}$ and $\mathcal{I}_{n,\ell,s,t}$, $\mathcal{F}_{n,k,t,r}$ and $\mathcal{F}_{n,\ell,t,r}$ are all maximal nontrivial cross-$t$-intersecting families. We also set
{\small 
\begin{align}
& h_{n,k,\ell,s,t}=|\mathcal{H}_{n,k,s,t}||\mathcal{I}_{n,\ell,s,t}|=\left({n-t\choose k-t}+t{n-s\choose k-s+1}\right)\left ({n-t\choose \ell-t}-{n-s\choose \ell-t}\right),\\
& f_{n,k,\ell,t,r}=|\mathcal{F}_{n,k,t,r}||\mathcal{F}_{n,\ell,t,r}|=\sum_{t+r\leq i,j\leq t+2r}{t+2r\choose i}{n-t-2r\choose k-i}{t+2r\choose j}{n-t-2r\choose \ell-j},\\
&g_{n,k,\ell,s,t}=|\mathcal{A}_{n,k,s,t}||\mathcal{B}_{n,\ell,s,t}|={n-s\choose \ell-s}\sum_{t\leq i\leq s}{s\choose i}{n-s\choose k-i}.
\end{align}}

In 2026, He et al. used the shift operator and  generating set method to determine the nontrivial extremal families for $t\geq 3$ as follows.

\begin{theorem}[He et al. \cite{HLWZ2026}]
Let $n,k,\ell$ and $t$ be positive integers with $k\geq \ell\geq t\geq 3$ and $n \geq (t+1)(k-t+1)$. 
If \(\mathcal{A}\subset \binom{[n]}{k}\) and \(\mathcal{B}\subset \binom{[n]}{\ell}\) are nontrivial cross-$t$-intersecting families, then
$$|\mathcal{A}||\mathcal{B}|\leq \max\{h_{n,\ell,k,t+1,t}, h_{n,\ell,k,\ell+1,t},f_{n,k,\ell,t,1}\}.$$
Moreover, the equality holds if and only if, up to isomorphism, one of the following holds:
\begin{itemize}
	\item [(i)] $\mathcal{A}=\mathcal{I}_{n,k,t+1,t}$ and $\mathcal{B}=\mathcal{H}_{n,\ell,t+1,t}$ when $h_{n,\ell,k,t+1,t}\geq \max\{ h_{n,\ell,k,\ell+1,t},f_{n,k,\ell,t,1}\}$;
	
	\item [(ii)] $\mathcal{A}=\mathcal{I}_{n,k,\ell+1,t}$ and $\mathcal{B}=\mathcal{H}_{n,\ell,\ell+1,t}$ when $h_{n,\ell,k,\ell+1,t}\geq \max\{h_{n,\ell,k,t+1,t},f_{n,k,\ell,t,1}\}$;
	
	\item [(iii)] $\mathcal{A}=\mathcal{F}_{n,k,t,1}$ and $\mathcal{B}=\mathcal{F}_{n,\ell,t,1}$ when $f_{n,k,\ell,t,1}\geq \max\{h_{n,\ell,k,t+1,t}, h_{n,\ell,k,\ell+1,t}\}$.
\end{itemize}

\end{theorem}

Similar to solving Tokushige's conjecture on non-uniform cross-$t$-intersecting families for $t=2$ \cite{BJX2026}, we also apply the shift operator and the generating set method to determine the nontrivial extremal cross-$2$-intersecting families in this paper.

\begin{theorem}\label{MainResult}
Let $n,k$ and $\ell$ be positive integers with $n \geq 3(k-1)$ and $k\geq \ell$. Suppose that 
\(\mathcal{A}\subset \binom{[n]}{k}\) and \(\mathcal{B}\subset \binom{[n]}{\ell}\) are nontrivial cross-$2$-intersecting families. If $\ell\geq 3$, then
$$|\mathcal{A}||\mathcal{B}|\leq \max\{h_{n,\ell,k,3,2}, h_{n,\ell,k,\ell+1,2},f_{n,k,\ell,2,1}\},$$
moreover, the equality holds if and only if, up to isomorphism, one of the following holds:
\begin{itemize}
	\item [(i)] $\mathcal{A}=\mathcal{I}_{n,k,3,2}$ and $\mathcal{B}=\mathcal{H}_{n,\ell,3,2}$ when $h_{n,\ell,k,3,2}\geq \max\{ h_{n,\ell,k,\ell+1,2},f_{n,k,\ell,2,1}\}$;
	
	\item [(ii)] $\mathcal{A}=\mathcal{I}_{n,k,\ell+1,2}$ and $\mathcal{B}=\mathcal{H}_{n,\ell,\ell+1,2}$ when $h_{n,\ell,k,\ell+1,2}\geq \max\{h_{n,\ell,k,3,2},f_{n,k,\ell,2,1}\}$;
	
	\item [(iii)] $\mathcal{A}=\mathcal{F}_{n,k,2,1}$ and $\mathcal{B}=\mathcal{F}_{n,\ell,2,1}$ when $f_{n,k,\ell,2,1}\geq \max\{h_{n,\ell,k,3,2}, h_{n,\ell,k,\ell+1,2}\}$.
\end{itemize}
If $\ell=2$, then $k\geq 3$ and $|\mathcal{A}||\mathcal{B}|\leq 3\binom{n-3}{k-3}$, moreover, the equality holds if and only if, up to isomorphism, $\mathcal{B}=\binom{[3]}{2}$ and $\mathcal{A}=\{A\in \binom{[n]}{k}\colon [3]\subset A\}$. 
	
\end{theorem}

 The rest of this paper is organized as follows: Section 2 introduces the necessary definitions and fundamental properties of the methods. Section 3 presents crucial inequalities and Section 4 provides a complete proof of our main result.  

\section{Preliminaries}
Let $\mathcal{A}$ be a family consisting of $k$-element subsets of $[n]$. For $i,j\in[n]$ and $A\in\mathcal{A}$, define 
\[
S_{ij}(A)= 
\begin{cases}
(A\backslash\{j\})\cup\{i\} & \text{if } j\in A, i\notin A,(A\backslash\{j\})\cup\{i\}\notin \mathcal{A}, \\
A & \text{otherwise,}
\end{cases}
\]
and set $S_{ij}(\mathcal{A}) = \{S_{ij}(A):A\in\mathcal{A}\}$ correspondingly. The procedure to obtain $S_{ij}(\mathcal{A})$ from $\mathcal{A}$ is called the {\em shift operation}, which was first introduced in \cite{EKR1961} (see also \cite{F1987}). We observe that $S_{ij}(\mathcal{A})$ has the same cardinality as $\mathcal{A}$ and is also $k$-uniform. We say that $\mathcal{A}$ is {\em left-compressed} if $S_{ij}(\mathcal{A})=\mathcal{A}$ for all ordered pairs $(i,j)$ with $1\leq i < j\leq n$. The following fact is well-known. 

\begin{fact}[\cite{F1987}] \label{lcl}
Let $\mathcal{A}$ and $\mathcal{B}$ be two families of subsets of $[n]$. If $\mathcal{A}$ and $\mathcal{B}$ are cross-$t$-intersecting, then $S_{i,j}(\mathcal{A})$ and $S_{i,j}(\mathcal{B})$ are cross-$t$-intersecting with $|S_{i,j}(\mathcal{A})| = |\mathcal{A}|$ and $|S_{i,j}(\mathcal{B})| = |\mathcal{B}|$.
\end{fact}

\begin{lemma}[He et al. \cite{HLWZ2026}]\label{Sij}
Let \(\mathcal{A}\subseteq \binom{[n]}{k}\) and \(\mathcal{B}\subseteq \binom{[n]}{\ell}\) be two maximal nontrivial cross-$t$-intersecting families. One may apply certain $S_{ij}$ operations on \(\mathcal{A}\) and \(\mathcal{B}\) such that the resulting families are left-compressed, nontrivial, and cross-$t$-intersecting.
\end{lemma}

For a subset $E$ of $[n]$, we define 
{\small \[s^{+}(E)=\max\{i:i\in E\},\ \mathcal{U}(E)=\{A\subseteq[n]:E\subseteq A\}, \ \mathcal{D}(E)=\left\{B\in\binom{[n]}{k}:B\cap[s^{+}(E)] = E\right\}.\] }
For a family $\mathcal{E}\subseteq 2^{[n]},$ we set 
{\small \[s^{+}(\mathcal{E})=\max\{s^{+}(E):E\in\mathcal{E}\},\   \mathcal{U}(\mathcal{E})=\bigcup_{E\in\mathcal{E}}\mathcal{U}(E), \ \mathcal{D}(\mathcal{E})=\bigcup_{E\in\mathcal{E}}\mathcal{D}(E).\] } 
For $\emptyset\neq \mathcal{F}\subset \binom{[n]}{k}$, a family $g(\mathcal{F})\subseteq\bigcup_{i\leq k}\binom{[n]}{i}$ is called a {\em generating set} of $\mathcal{F}$ if 
{\small \[\mathcal{U}(g(\mathcal{F}))\cap\binom{[n]}{k}=\mathcal{F}.\]}
Clearly, $\mathcal{F}$ itself is a generating set of itself. The set of all generating sets of $\mathcal{F}$ is nonempty and denoted by $G(\mathcal{F})$. The notion of a generating set of a $k$-uniform family was first introduced in \cite{AK1997} (see also \cite{AK1996}). For a generating set $g(\mathcal{F})\in G(\mathcal{F})$, let $g_{*}(\mathcal{F})$ be the set of all minimal elements of $g(\mathcal{F})$ with respect to set inclusion. Define
\(G_{*}(\mathcal{F})=\{g(\mathcal{F})\in G(\mathcal{F}):g(\mathcal{F}) = g_{*}(\mathcal{F})\}.\) Clearly, $G_{*}(\mathcal{F})\neq \emptyset$. Take any $g(\mathcal{F})\in G_{*}(\mathcal{F})$ and set 
\[s = s^{+}(g(\mathcal{F})).\]
Then define
{\small \[g^{*}(\mathcal{F})=\{E\in g(\mathcal{F}):s\in E\},\ g_{i}^{*}(\mathcal{F})=\{E\in g^{*}(\mathcal{F}):|E| = i\}\ {\rm for}\ 1\leq i\leq s\]}
and 
{\small \[g_{i}^{*}(\mathcal{F})'=\{E\backslash\{s\}:E\in g_{i}^{*}(\mathcal{F})\}.\] }

From \cite{AK1997}, we know that the generating sets have the following properties.

\begin{lemma}[Ahlswede and Khachatrian \cite{AK1997}]\label{AKGS}
Let \(\mathcal{F}\) be a left-compressed \(t\)-intersecting subfamily of \(\binom{[n]}{k}\), \(g(\mathcal{F}) \in G_{*}(\mathcal{F})\) and \(s = s^{+}(g(\mathcal{F}))\). Then the following statements hold.

\begin{enumerate}[(i)]
    \item If \(n > 2k - t\), then \(|E_1 \cap E_2| \geq t\) for all \(E_1, E_2 \in g(\mathcal{F})\).

    \item For \(1 \leq i < j \leq s\) and \(E \in g(\mathcal{F})\), one has either \(S_{ij}(E) \in g(\mathcal{F})\) or \(F \subset S_{ij}(E)\) for some \(F \in g(\mathcal{F})\).

    \item \(\mathcal{F}\) is a disjoint union \(\mathcal{F} = \bigcup_{E \in g(\mathcal{F})} \mathcal{D}(E)\).

    \item If \(\mathcal{F}\) is maximal, then for any \(E_1, E_2 \in g^{*}(\mathcal{F})\) with \(|E_1 \cap E_2| = t\), necessarily \(|E_1| + |E_2| = s + t\) and \(E_1 \cup E_2 = [s]\). Furthermore, if \(g^{*}_{i}(\mathcal{F}) \neq \emptyset\), then \(g^{*}_{s + t - i}(\mathcal{F}) \neq \emptyset\) and for any \(E_1 \in g^{*}_{i}(\mathcal{F})\), there exists \(E_2 \in g^{*}_{s + t - i}(\mathcal{F})\) with \(|E_1 \cap E_2| = t\) and \(E_1 \cup E_2 = [s]\).

    \item If \(g^{*}_{i}(\mathcal{F}) \neq \emptyset\), then \(\mathcal{F}_1 = \mathcal{F} \cup \mathcal{D}(g^{*}_{i}(\mathcal{F})) \setminus \mathcal{D}(g^{*}_{s + t - i}(\mathcal{F}))\) is also a \(t\)-intersecting subfamily of \(\binom{[n]}{k}\) with
    \[
    |\mathcal{F}_1| = |\mathcal{F}| + |g^{*}_{i}(\mathcal{F})| \binom{n - s}{k - i + 1} - |g^{*}_{s + t - i}(\mathcal{F})| \binom{n -s}{k + i -s - t}.
    \]
\end{enumerate}
\end{lemma}

Let \(\mathcal{A}\subseteq \binom{[n]}{k}\) and \(\mathcal{B}\subseteq \binom{[m]}{\ell}\) be two cross-$t$-intersecting families. For $\min\{m, n\}>k+\ell-t$, it is straightforward to verify that $\mathcal{A}$ and $\mathcal{B}$ are cross-$t$-intersecting if and only if $g(\mathcal{A})$ and $g(\mathcal{B})$ are cross-$t$-intersecting for all $g(\mathcal{A})\in G(\mathcal{A})$ and $g(\mathcal{B})\in G(\mathcal{B})$. By applying the shift operations if necessary, we may assume without loss of generality that both $\mathcal{A}$ and $\mathcal{B}$ are left-compressed. Zhang and Wu \cite{ZW2025} obtained properties for the generating sets of families $\mathcal{A}, \mathcal{B} \subseteq \binom{[n]}{k}$. In a similar vein, the first two authors derived the following result. 

\begin{lemma}[Bao and Ji \cite{BJ2025}]\label{lemkl}
Let \(\mathcal{A}\subseteq \binom{[n]}{k}\) and \(\mathcal{B}\subseteq \binom{[m]}{\ell}\) be two maximal left-compressed cross-$t$-intersecting families with $\min\{m, n\}>k+\ell-t$, $g(\mathcal{A})\in G_{*}(\mathcal{A})$, $g(\mathcal{B})\in G_{*}(\mathcal{B})$ such that $s:=\max\{s^{+}(g(\mathcal{A})),s^{+}(g(\mathcal{B}))\}$ is minimal. Then the following statements hold.
\begin{enumerate}[(i)]
    \item For $1\leq i<j\leq s$, $\mathcal{F}\in\{\mathcal{A},\mathcal{B}\}$ and $E\in g(\mathcal{F})$, one has either $S_{ij}(E)\in g(\mathcal{F})$ or $F\subset S_{ij}(E)$ for some $F\in g(\mathcal{F})$.
    \item For $t\leq i\leq k$, $g_{i}^{*}(\mathcal{A})\neq\emptyset$ if and only if $g_{s + t - i}^{*}(\mathcal{B})\neq\emptyset$. Furthermore, for each $E\in g_{i}^{*}(\mathcal{A})$, there exists $F\in g_{s + t - i}^{*}(\mathcal{B})$ such that $|E\cap F| = t$ and $E\cup F=[s]$.
    \item If $g_{i}^{*}(\mathcal{A})\neq\emptyset$, then $\mathcal{A}_{1}=\mathcal{A}\cup \mathcal{D}(g_{i}^{*}(\mathcal{A})')\subseteq \binom{[n]}{k}$ and $\mathcal{B}_{1}=\mathcal{B}\setminus \mathcal{D}(g_{s + t - i}^{*}(\mathcal{B}))\subseteq \binom{[m]}{\ell}$ are also cross-$t$-intersecting families with 
    \[|\mathcal{A}_{1}|=|\mathcal{A}|+|g_{i}^{*}(\mathcal{A})|\binom{n - s}{k - i+1}\quad\text{and}\quad|\mathcal{B}_{1}|=|\mathcal{B}|-|g_{s + t - i}^{*}(\mathcal{B})|\binom{m-s}{\ell + i - s - t}.\]
\end{enumerate}
\end{lemma}

Similar to \cite[Lemma 3.6]{HLWZ2026}, we have the following result.

\begin{lemma}\label{transform}
	Let \(\mathcal{A}\subseteq \binom{[n]}{k}\) and \(\mathcal{B}\subseteq \binom{[n]}{\ell}\) be two maximal left-compressed nontrivial cross-$t$-intersecting families with $n\geq (t+1)(k-t+1)$ and $ k\geq \ell\geq t\geq 2$, $g(\mathcal{A})\in G_{*}(\mathcal{A})$, $g(\mathcal{B})\in G_{*}(\mathcal{B})$ such that $s:=\max\{s^{+}(g(\mathcal{A})),s^{+}(g(\mathcal{B}))\}$ is minimal. If $g_i^*(\mathcal{A})\neq \emptyset$, then $\mathcal{A}_{1}=\mathcal{A}\cup \mathcal{D}(g_{i}^{*}(\mathcal{A})')\subseteq \binom{[n]}{k}$ and $\mathcal{B}_{1}=\mathcal{B}\setminus \mathcal{D}(g_{s + t - i}^{*}(\mathcal{B}))\subseteq \binom{[n]}{\ell}$ are trivial if and only if $\mathcal{B}=\mathcal{H}_{n,\ell,s,t}$. 
\end{lemma}	

\proof Let $\mathcal{A}_{1}=\mathcal{A}\cup \mathcal{D}(g_{i}^{*}(\mathcal{A})')$ and $\mathcal{B}_{1}=\mathcal{B}\backslash \mathcal{D}(g_{\mathrm{s}+t-i}^{*}(\mathcal{B}))$. Clearly, $ g(\mathcal{A})\cup(g_{i}^{*}(\mathcal{A}))'\in G(\mathcal{A}_{1})$ and $g(\mathcal{B})\backslash g_{\mathrm{s}+t-i}^{*}(\mathcal{B})\in G(\mathcal{B}_{1})$. Suppose that $\mathcal{A}$ and $\mathcal{B}$ are nontrivial, but $\mathcal{A}_{1}$ and $\mathcal{B}_{1}$ are trivial. Then $[t]\subset F$ for all $F\in (g(\mathcal{A})\cup g(\mathcal{B})$)$\backslash g_{\mathrm{s}+t-i}^{*}(\mathcal{B})$. By the assumption that $\mathcal{A}$ and $\mathcal{B}$ are nontrivial maximal cross-$t$-intersecting, it follows that $[t]\in g(\mathcal{B})$. Therefore, $[t]\not\subset F$ for each $F\in g_{s+t-i}^{*}(\mathcal{B})$. Moreover, for $F\in g_{s+t-i}^{*}(\mathcal{B})$, by (i) of Lemma \ref{lemkl}, we have $[t]\subset S_{js}(F)$ for each $j\in[s]\backslash F$, which implies $F=[s]\backslash \{j\}$ for some $j\in[t]$. Therefore, $i=t+1$, $g(\mathcal{A})=\{E\in \binom{[s]}{t+1}\colon [t]\subset E\}$ and $g(\mathcal{B})=\{E\in \binom{[s]}{s-1}\colon |E\cap [t]|=t-1\}\cup \{[t]\}$. 
It follows that $\mathcal{A}=\mathcal{I}_{n,k,s,t}$ and $\mathcal{B}=\mathcal{H}_{n,\ell,s,t}$. 

Suppose that $\mathcal{B}=\mathcal{H}_{n,\ell,s,t}$ for some $s\leq \ell+1$. By the assumption that $\mathcal{A}$ and $\mathcal{B}$ are maximal nontrivial cross-$t$-intersecting families, we have $\mathcal{A}=\mathcal{I}_{n,k,s,t}$. Then $\mathcal{A}_1=\{A\in \binom{[n]}{k}\colon [t]\subset A\}$ and $\mathcal{B}_1=\{B\in \binom{[n]}{\ell}\colon [t]\subset B\}$, 
that is, $\mathcal{A}_1$ and $\mathcal{B}_1$ are trivial. This completes the proof. \qed



\section{Proofs of some basic inequalities}

In this section, we give some basic inequalities, which are crucial to our proof of Theorem \ref{MainResult}. 

The following inequality will be used frequently, its proof is straightforward and omitted here.

\begin{lemma}\label{trivialInEq}
Let $x,y,z,u,a,b$ be positive real numbers with $xu-a\geq 0, yu-b>0$ and $z\geq u$. Then $\frac{xz-a}{yz-b}\geq \frac{xu-a}{yu-b}$ if and only if  $ay\geq bx$, and $\frac{xz-a}{yz-b}\leq \frac{xu-a}{yu-b}$ if and only if  $ay\leq bx$. 
\end{lemma}

\begin{lemma}\label{hn,l,k,3,2}
	Let $n,k$ and $\ell$ be integers with $k> \ell\geq 3$ and $n\geq 3k-3$. Then 
\begin{align*}
& h_{n,\ell,k,3,2}>	h_{n,k,\ell,3,2},\\
& h_{n,\ell,k,\ell+1,2}> 	h_{n,k,\ell,k+1,2}.
\end{align*}	
\end{lemma}	

\proof Since $3\leq \ell< k$, by the definition of $h_{n,k,\ell,s,2}$ in (1.1), we have 
\begin{align*}
	\displaystyle \frac{h_{n,k,\ell,3,2}}{h_{n,\ell,k,3,2}}&=\frac{(2\binom{n - 3}{k - 2}+\binom{n - 2}{k - 2})\binom{n - 3}{\ell - 3}}{(2\binom{n - 3}{\ell - 2}+\binom{n - 2}{\ell - 2})\binom{n - 3}{k - 3}}=\frac{\frac{2(n-k)}{k-2}+\frac{n-2}{k-2}}{\frac{2(n-\ell)}{\ell-2}+\frac{n-2}{\ell-2}}=\frac{\frac{3(n-2)}{k-2}-2}{\frac{3(n-2)}{\ell-2}-2}< 1.
\end{align*}
Hence $ h_{n,k,\ell,3,2}< h_{n,\ell,k,3,2}$. 

Also, by the definition of $h_{n,k,\ell,s,2}$ in (1.1), we have
\begin{align*}
	&	h_{n,\ell,k,\ell+1,2}-	h_{n,k,\ell,k+1,2}\\
	&={\small \left (\binom{n-2}{\ell-2}+2\right)\left (\binom{n-2}{k-2}-\binom{n-\ell-1}{k-2}\right)-\left (\binom{n-2}{k-2}+2\right)\left (\binom{n-2}{\ell-2}-\binom{n-k-1}{\ell-2}\right)}\\
	& =2\left(\binom{n-2}{k-2}-\binom{n-2}{\ell-2}-\binom{n-\ell-1}{k-2}+\binom{n-k-1}{\ell-2} \right)+\binom{n-2}{k-2}\binom{n-k-1}{\ell-2}\\
	& \quad -\binom{n-2}{\ell-2}\binom{n-\ell-1}{k-2}.
\end{align*}

For $\ell=3$, we have 
\begin{align*}
	&	h_{n,\ell,k,\ell+1,2}-	h_{n,k,\ell,k+1,2}\\
& =(n-k+1)\binom{n-2}{k-2}-n\binom{n-4}{k-2}-2\binom{n-2}{1}+2\binom{n-k-1}{1} \\
& >\binom{n-4}{k-2}+\frac{(n-2)(n-3)}{n-k-1}\binom{n-4}{k-2}-n\binom{n-4}{k-2}-2(k-1)\\
& >\binom{n-4}{k-2}-2(k-1)\geq \binom{3k-7}{2}-2(k-1)=\frac{(9k-40)(k-1)+20}{2} >0.
\end{align*}

For $\ell\geq 4$, since 
\begin{align*}
\binom{n-2}{k-2}\binom{n-k-1}{\ell-2}& =\frac{n-\ell}{n-k} \binom{n-2}{\ell-2}\binom{n-\ell-1}{k-2}\\
& \geq \binom{n-2}{\ell-2}\binom{n-\ell-1}{k-2} + \frac{1}{n-k} \binom{n-2}{\ell-2}\binom{n-\ell-1}{k-2}\\
& \geq \binom{n-2}{\ell-2}\binom{n-\ell-1}{k-2} + \frac{1}{n-k} \binom{n-2}{2}\binom{n-\ell-1}{k-2}\\
& \geq \binom{n-2}{\ell-2}\binom{n-\ell-1}{k-2} + \frac{n-2}{2} \binom{n-\ell-1}{k-2}\\
& \geq \binom{n-2}{\ell-2}\binom{n-\ell-1}{k-2} + \frac{3k-5}{2} \binom{n-\ell-1}{k-2}\\
& \geq \binom{n-2}{\ell-2}\binom{n-\ell-1}{k-2} + 2 \binom{n-\ell-1}{k-2},
\end{align*}
we have 
\begin{align*}
	&	h_{n,\ell,k,\ell+1,2}-	h_{n,k,\ell,k+1,2}\\
&=2\left(\binom{n-2}{k-2}-\binom{n-2}{\ell-2}+\binom{n-k-1}{\ell-2} \right)+\binom{n-2}{k-2}\binom{n-k-1}{\ell-2}\\
	& \quad -\binom{n-2}{\ell-2}\binom{n-\ell-1}{k-2}-2\binom{n-\ell-1}{k-2}\\
	& >0. 
\end{align*}

This completes the proof. \qed

\begin{lemma}\label{s=4}
Let $n,k$ and $\ell$ be positive integers with $k\geq 4$, $\ell\geq 3$ and $n\geq \max\{3k-3,3\ell-3\}$. Then $h_{n,k,\ell,4,2}< 	h_{n,k,\ell,k+1,2}$.
\end{lemma}	

\proof By the definition of $h_{n,k,\ell,s,2}$ in (1.1), we have 
\begin{align*}
	h_{n,k,\ell,k+1,2}-	h_{n,k,\ell,4,2}&=
 \left (\binom{n-2}{k-2}+2\right)\left (\binom{n-2}{\ell-2}-\binom{n-k-1}{\ell-2}\right)\\ & \quad -\left (\binom{n-2}{k-2}+2\binom{n-4}{k-3}\right)\left (\binom{n-2}{\ell-2}-\binom{n-4}{\ell-2}\right)\\
& = \binom{n-2}{k-2} \left(\binom{n-4}{\ell-2}-\binom{n-k-1}{\ell-2}\right)+2 \left(\binom{n-2}{\ell-2}-\binom{n-k-1}{\ell-2}\right)\\
& \quad -2\binom{n-4}{k-3}\left (\binom{n-2}{\ell-2}-\binom{n-4}{\ell-2}\right).
\end{align*}
We distinguish according to values of $k,\ell$. 

Case 1: $k=4$ and $\ell \geq 6$. Simple computation shows that 
\begin{align*}
 &	h_{n,k,\ell,k+1,2}-	h_{n,k,\ell,4,2}\\
 & = \binom{n-2}{2}\binom{n-5}{\ell-3}+2\binom{n-3}{\ell-3}+2\binom{n-4}{\ell-3}+2\binom{n-5}{\ell-3}-2(n-4)\left( \binom{n-3}{\ell-3}+\binom{n-4}{\ell-3}\right)\\
& =\left( \frac{n^2-5n+10}{2}-2(n-5)\left( \frac{(n-3)(n-4)}{(n-\ell)(n-\ell-1)}+\frac{n-4}{n-\ell-1}\right )  \right )\binom{n-5}{\ell-3}\\
& \geq \left( \frac{n^2-5n+10}{2}-2(n-5)\left( \frac{(3\ell-6)(3\ell-7)}{(
	2\ell-3)(2\ell-4)}+\frac{3\ell-7}{2\ell-4}\right )  \right )\binom{n-5}{\ell-3}\\
& \geq \left( \frac{n^2-5n+10}{2}-2(n-5)\cdot (\frac{3}{2}\cdot \frac{3}{2}+\frac{3}{2})\right ) \binom{n-5}{\ell-3}\\
& >0,
\end{align*}
where the last inequality holds because $n\geq \max\{3k-3,3\ell-3\}\geq 15$, and the other two inequalities hold by applying Lemma \ref{trivialInEq}.

Case 2: $k=4$ and $\ell = 5$. Since $n\geq \max\{3k-3,3\ell-3\}\geq 12$, we have that
\begin{align*}
	&	h_{n,k,\ell,k+1,2}-	h_{n,k,\ell,4,2}\\
	& = \binom{n-2}{2}\binom{n-5}{2}+2\binom{n-3}{2}+2\binom{n-4}{2}+2\binom{n-5}{2}-2(n-4)\left( \binom{n-3}{2}+\binom{n-4}{2}\right)\\
	& = \frac{(n-2)(n-3)(n-5)(n-6)-8(n-4)^3}{4}+2\binom{n-3}{2}+2\binom{n-4}{2}+2\binom{n-5}{2} \\
	& >0.
\end{align*}

Case 3: $k=4$ and $\ell = 4$. Since $n\geq \max\{3k-3,3\ell-3\}\geq 9$, we have that
\begin{align*}
	&	h_{n,k,\ell,k+1,2}-	h_{n,k,\ell,4,2}\\
	& = \binom{n-2}{2}(n-5)+2\binom{n-3}{1}+2\binom{n-4}{1}+2\binom{n-5}{1}-2(n-4)(2n-7)\\
	& = \frac{n^3-18n^2+103n-190}{2}\\
	& >0.
\end{align*}

Case 4: $k=4$ and $\ell = 3$. Since $n\geq \max\{3k-3,3\ell-3\}\geq 9$, we have that
\begin{align*}
	&	h_{n,k,\ell,k+1,2}-	h_{n,k,\ell,4,2} = \binom{n-2}{2}+6-2(n-4)\cdot 2= \frac{n^2-13n+50}{2} >0.
\end{align*}

Case 5: $k=5$ and $\ell \geq 6$. Since $n\geq \max\{3k-3,3\ell-3\}\geq 15$, we have that
\begin{align*}
\frac{\binom{n-2}{3}}{2\binom{n-4}{2}}=\frac{(n-2)(n-3)}{6(n-5)}\geq \frac{13\cdot 12}{6\cdot 10}=\frac{13}{5}.
\end{align*}
By Lemma \ref{trivialInEq}, we have that 
\begin{align*}
& \frac{13(n-\ell)(n-\ell-1)(2n-\ell-7)}{5(n-4)(n-5)(2n-\ell-3)}\geq \frac{13(2\ell-3)(2\ell-4)(5\ell-13)}{5(3\ell-7)(3\ell-8)(5\ell-9)}.
\end{align*}
Simple computation shows that 
\begin{align*}
	& 13(2\ell-3)(2\ell-4)(5\ell-13)-5(3\ell-7)(3\ell-8)(5\ell-9)=35\ell^3-56\ell^2-279\ell+492>0.
\end{align*}
It follows that $13(n-\ell)(n-\ell-1)(2n-\ell-7)>5(n-4)(n-5)(2n-\ell-3)$ and 

\begin{align*}
& 13 \left (\binom{n-5}{\ell-3}+\binom{n-6}{\ell-3}\right)-5 \left (\binom{n-3}{\ell-3}+\binom{n-4}{\ell-3}\right)	\\
&=  \binom{n-6}{\ell-3} \left( 13+ 13\cdot \frac{n-5}{n-\ell-2}-5\cdot \frac{(n-3)(n-4)(n-5)}{(n-\ell)(n-\ell-1)(n-\ell-2)}-5\cdot \frac{(n-4)(n-5)}{(n-\ell-1)(n-\ell-2)}\right)\\
&=  \binom{n-6}{\ell-3} \cdot \frac{13(n-\ell)(n-\ell-1)(2n-\ell-7)-5(n-4)(n-5)(2n-\ell-3)}{(n-\ell)(n-\ell-1)(n-\ell-2)}\\
&>0.
\end{align*}
Hence, 
\begin{align*}
& h_{n,k,\ell,k+1,2}-	h_{n,k,\ell,4,2}\\
&= \binom{n-2}{3}\left (\binom{n-5}{\ell-3}+\binom{n-6}{\ell-3}\right)+2\left(\binom{n-2}{\ell-2}-\binom{n-6}{\ell-2}\right)-2\binom{n-4}{2} \left (\binom{n-3}{\ell-3}+\binom{n-4}{\ell-3}\right)\\
& \geq 2\binom{n-4}{2} \left(\frac{13}{5}\left(\binom{n-5}{\ell-3}+\binom{n-6}{\ell-3}\right)-\left(\binom{n-3}{\ell-3}+\binom{n-4}{\ell-3}\right)\right)+2\left(\binom{n-2}{\ell-2}-\binom{n-6}{\ell-2}\right) \\
&>0
\end{align*}

Case 6: $k=5$ and $\ell =3$. Since $n\geq \max\{3k-3,3\ell-3\}\geq 12$, we have that
\begin{align*}
	& h_{n,k,\ell,k+1,2}-	h_{n,k,\ell,4,2}=2\binom{n-2}{3}+2\cdot 4-2\binom{n-4}{2}\cdot 2>0.
\end{align*}

Case 7: $k=5$ and $\ell =4$. Since $n\geq \max\{3k-3,3\ell-3\}\geq 12$, we have that
\begin{align*}
	& h_{n,k,\ell,k+1,2}-	h_{n,k,\ell,4,2}\\
	& =(2n-11)\binom{n-2}{3}+2\left(\binom{n-2}{2}-\binom{n-6}{2}\right) -2(2n-7)\binom{n-4}{2}\\
	& >(2n-11)\binom{n-2}{3}-2(2n-7)\binom{n-4}{2}\\
	& = \frac{(n-4)}{6}\cdot (2n^3-33n^2+169n-276)\\
	& >0.
\end{align*}

Case 8: $k=5$ and $\ell =5$. Since $n\geq \max\{3k-3,3\ell-3\}\geq 12$, we have that
\begin{align*}
	& h_{n,k,\ell,k+1,2}-	h_{n,k,\ell,4,2}\\
	&  =\binom{n-2}{3}\left (\binom{n-5}{2}+\binom{n-6}{2}\right)+2\left(\binom{n-2}{3}-\binom{n-6}{3}\right)-2\binom{n-4}{2} \left (\binom{n-3}{2}+\binom{n-4}{2}\right)\\
	& >\binom{n-2}{3}\left (\binom{n-5}{2}+\binom{n-6}{2}\right)-2\binom{n-4}{2} \left (\binom{n-3}{2}+\binom{n-4}{2}\right)\\
	& = \frac{(n-4)}{6}\cdot (n^4-23n^3+180n^2-588n+696)\\
	& >0.
\end{align*}

Case 9: $k\geq 6$. Then 
\begin{align*}
	& h_{n,k,\ell,k+1,2}-	h_{n,k,\ell,4,2}\\
	&  =\binom{n-2}{k-2}\left (\binom{n-5}{\ell-3}+\binom{n-6}{\ell-3}+\cdots+\binom{n-k-1}{\ell-3}\right)+2\left(\binom{n-2}{\ell-2}-\binom{n-k-1}{\ell-2}\right)\\
	& \quad -2\binom{n-4}{k-3} \left (\binom{n-3}{\ell-3}+\binom{n-4}{\ell-3}\right)\\
	& >\binom{n-2}{k-2}\left (\binom{n-5}{\ell-3}+\binom{n-6}{\ell-3}+\binom{n-7}{\ell-3}\right)-2\binom{n-4}{k-3} \left (\binom{n-3}{\ell-3}+\binom{n-4}{\ell-3}\right).
\end{align*}
Since $n\geq \max\{3k-3,3\ell-3\}\geq 15$,  by Lemma \ref*{trivialInEq} we have that
\begin{align*}
	\frac{\binom{n-2}{k-2}}{2\binom{n-4}{k-3}}=\frac{(n-2)(n-3)}{2(k-2)(n-k)}\geq \frac{(3k-5)(3k-6)}{2(k-2)(2k-3)}\geq \frac{13}{6}.
\end{align*}
Consider 
\begin{align*}
	& g(n,\ell)=13\left (\binom{n-5}{\ell-3}+\binom{n-6}{\ell-3}+\binom{n-7}{\ell-3}\right)-6\left (\binom{n-3}{\ell-3}+\binom{n-4}{\ell-3}\right). 
\end{align*}
Since $n\geq 15$, simple computation shows that 
\begin{align*}
& g(n,3)>0, g(n,4)=27n-192>0, g(n,5)=\frac{27n^2-411n+1472}{2}>0,\\
& g(n+1,\ell)=g(n,\ell)+g(n,\ell-1).
\end{align*}
For $\ell \geq 6$ and $n\geq 3\ell-3$, we have 
\begin{align*}
	 g(3\ell-3,\ell)& =13\left (\binom{3\ell-8}{\ell-3}+\binom{3\ell-9}{\ell-3}+\binom{3\ell-10}{\ell-3}\right)-6\left (\binom{3\ell-6}{\ell-3}+\binom{3\ell-7}{\ell-3}\right). \\
	 &= \binom{3\ell-6}{\ell-3}\left( \frac{26}{3} \cdot \frac{2\ell-3}{3\ell-7}\cdot  \frac{19\ell-49}{9\ell-24}-6\cdot  \frac{5\ell-9}{3\ell-6} \right)\\
	 &= \binom{3\ell-6}{\ell-3} \frac{534\ell^3-1494\ell^2-1344\ell+4284}{9(3\ell-6)(3\ell-7)(3\ell-8)}\\
	 &>0.
\end{align*}
It follows from inductive method that $g(n,\ell)>0$ for all $n\geq 3\ell-3$. Hence $ h_{n,k,\ell,k+1,2}>h_{n,k,\ell,4,2}$. This completes the proof. \qed

\begin{lemma}\label{4<s<k} 
	Let $n,k$ and $\ell$ be positive integers with $k\geq 5$, $\ell\geq 3$ and $n\geq \max\{3k-3,3\ell-3\}$. Then $h_{n,k,\ell,s,2}< 	h_{n,k,\ell,k+1,2}$ for $5\leq s\leq k$. 
\end{lemma}	

\proof By the definition of $h_{n,k,\ell,s,2}$ in (1.1), we have 
\begin{align*}
	&	h_{n,k,\ell,k+1,2}-	h_{n,k,\ell,s,2}\\
	&=\left (\binom{n-2}{k-2}+2\right)\left (\binom{n-2}{\ell-2}-\binom{n-k-1}{\ell-2}\right)\\
	& \quad -\left (\binom{n-2}{k-2}+2\binom{n-s}{k-s+1}\right)\left (\binom{n-2}{\ell-2}-\binom{n-s}{\ell-2}\right).\\
	& = 2\left (\binom{n-2}{\ell-2}-\binom{n-k-1}{\ell-2}\right)+\binom{n-2}{k-2}\left(\binom{n-s}{\ell-2}-\binom{n-k-1}{\ell-2}\right)\\
	& \quad -2\binom{n-s}{k-s+1}\left (\binom{n-2}{\ell-2}-\binom{n-s}{\ell-2}\right).
\end{align*}
It suffices to prove that 
\begin{align*}
& \binom{n-2}{k-2}\left(\binom{n-s}{\ell-2}-\binom{n-k-1}{\ell-2}\right)>2\binom{n-s}{k-s+1}\left (\binom{n-2}{\ell-2}-\binom{n-s}{\ell-2}\right).
\end{align*}
We distinguish two cases according to the value $s$.

Case 1: $s\geq 6$. On one hand, we have that
\begin{align*}
 \frac{\binom{n-2}{k-2}}{2\binom{n-s}{k-s+1}}&=\frac{(n-2)(n-3)\cdots(n-s+1)}{2(n-k)(k-2)(k-3)\cdots (k-s+2)}\\
 & \geq \frac{(3k-5)(3k-6)\cdots (3k-s-2)}{2(2k-3)(k-2)\cdots (k-s+2)}\\
 & \geq \frac{3k-5}{2(2k-3)}3^{s-3}\geq 3^{s-3}\cdot \frac{1}{2} \cdot \frac{3\cdot 5-5}{2\cdot 5-3}=\frac{5}{7}3^{s-3}.
\end{align*}
Simple computation shows that 
\begin{align*}
\frac{\binom{n-2}{\ell-2}-\binom{n-s}{\ell-2}}{\binom{n-s}{\ell-2}-\binom{n-k-1}{\ell-2}}&= \frac{\binom{n-3}{\ell-3}+\binom{n-4}{\ell-3}+\cdots+\binom{n-s}{\ell-3}}{\binom{n-s-1}{\ell-3}+\cdots +\binom{n-k-1}{\ell-3}}\leq \frac{\binom{n-3}{\ell-3}+\binom{n-4}{\ell-3}+\cdots+\binom{n-s}{\ell-3}}{\binom{n-s-1}{\ell-3}}\\
& =\frac{(n-3)(n-4)\cdots(n-s)}{(n-\ell)(n-\ell-1)\cdots(n-\ell-s+3)}+\frac{(n-4)\cdots(n-s)}{(n-\ell-1)\cdots(n-\ell-s+3)}\\
& \quad +\cdots +\frac{n-s}{n-\ell-s+3}.
\end{align*} 
If $\ell\geq k$, then by Lemma \ref{trivialInEq} we have that 
\begin{align*}
	\frac{\binom{n-2}{\ell-2}-\binom{n-s}{\ell-2}}{\binom{n-s}{\ell-2}-\binom{n-k-1}{\ell-2}}& \leq \frac{(3\ell-6)(3\ell-7)\cdots(3\ell-s-3)}{(2\ell-3)(2\ell-4)\cdots(2\ell-s)}+\frac{(3\ell-7)\cdots(3\ell-s-3)}{(2\ell-4)\cdots(2\ell-s)}\\
	& \quad +\cdots +\frac{3\ell-s-3}{2\ell-s}\\
&= \frac{3\ell-6}{2\ell-3} \cdot \frac{3\ell-7}{2\ell-4} \cdot \frac{3\ell-8}{2\ell-5} \cdot \left(\frac{3\ell-9}{2\ell-6}\cdots \frac{3\ell-3-s}{2\ell-s}\right) \\
&\quad  + \frac{3\ell-7}{2\ell-4} \cdot \frac{3\ell-8}{2\ell-5} \cdot \left(\frac{3\ell-9}{2\ell-6}\cdots \frac{3\ell-3-s}{2\ell-s}\right) \\
& \quad + \frac{3\ell-8}{2\ell-5} \cdot \left(\frac{3\ell-9}{2\ell-6}\cdots \frac{3\ell-3-s}{2\ell-s}\right) \\
& \quad +\left(\frac{3\ell-9}{2\ell-6}\cdots \frac{3\ell-3-s}{2\ell-s}\right) +\cdots + \frac{3\ell-s-3}{2\ell-s}.
\end{align*} 
Since $\frac{3\ell-6-a}{2\ell-3-a}\leq \frac{3}{2}$ for $0\leq a\leq 3$ and $\frac{3\ell-6-a}{2\ell-3-a}\leq 2$ for $a\leq \ell$, we further have 
\begin{align*}
	\frac{\binom{n-2}{\ell-2}-\binom{n-s}{\ell-2}}{\binom{n-s}{\ell-2}-\binom{n-k-1}{\ell-2}}& \leq \frac{27}{8}\cdot 2^{s-5}+\frac{9}{4}\cdot 2^{s-5}+\frac{3}{2}\cdot 2^{s-5}+2^{s-5}+2^{s-6}+\cdots+2\\
&< (\frac{27}{8}+\frac{9}{4}+\frac{3}{2}+2)\cdot 2^{s-5}=\frac{73}{8}\cdot 2^{s-5}\\
&<\frac{5}{7}\cdot 3^{s-3}.
\end{align*} 
If $k\geq \ell$, then by Lemma \ref{trivialInEq} we have that
\begin{align*}
	\frac{\binom{n-2}{\ell-2}-\binom{n-s}{\ell-2}}{\binom{n-s}{\ell-2}-\binom{n-k-1}{\ell-2}}& \leq \frac{(3k-6)(3k-7)\cdots(3k-s-3)}{(3k-\ell-3)(3k-\ell-4)\cdots(3k-\ell-s)}+\frac{(3k-7)\cdots(3k-s-3)}{(3k-\ell-4)\cdots(3k-\ell-s)}\\
	& \quad +\cdots +\frac{3k-\ell-s-3}{3k-\ell-s}\\
& \leq \frac{(3k-6)(3k-7)\cdots(3k-s-3)}{(2k-3)(2k-4)\cdots(2k-s)}+\frac{(3k-7)\cdots(3k-s-3)}{(2k-4)\cdots(2k-s)}\\
& \quad +\cdots +\frac{3k-\ell-s-3}{2k-s}.
\end{align*} 
Similar discussion yields $	\frac{\binom{n-2}{\ell-2}-\binom{n-s}{\ell-2}}{\binom{n-s}{\ell-2}-\binom{n-k-1}{\ell-2}}<\frac{5}{7}\cdot 3^{s-3}$.  So, the conclusion holds.

Case 2: $s=5$. Since $n\geq 3k-3$, we have that
\begin{align*}
	\frac{\binom{n-2}{k-2}}{2\binom{n-s}{k-s+1}}&=\frac{(n-2)(n-3)(n-4)}{2(n-k)(k-2)(k-3)} \geq \frac{(3k-5)(3k-6)(3k-7)}{2(2k-3)(k-2)(k-3)}\geq \left \{\begin{array}{ll}
		\frac{60}{7}, & \text{if}\ k=5,\smallskip \\
		\frac{27}{4}, & \text{if}\ k\geq 6.\\
	\end{array} \right.
\end{align*}

For $k=5$, since $n\geq 3\ell-3$, by Lemma \ref{trivialInEq} we have that
\begin{align*}
\frac{\binom{n-2}{\ell-2}-\binom{n-5}{\ell-2}}{\binom{n-5}{\ell-2}-\binom{n-6}{\ell-2}}&=\frac{\binom{n-3}{\ell-3}+\binom{n-4}{\ell-3}+\binom{n-5}{\ell-3}}{\binom{n-6}{\ell-3}}\\ & =\frac{(n-3)(n-4)(n-5)}{(n-\ell)(n-\ell-1)(n-\ell-2)}+\frac{(n-4)(n-5)}{(n-\ell-1)(n-\ell-2)}+\frac{n-5}{n-\ell-2}\\ 
& \leq \frac{(3\ell-6)(3\ell-7)(3\ell-8)}{(2\ell-3)(2\ell-4)(2\ell-5)}+\frac{(3\ell-6)(3\ell-7)}{(2\ell-4)(2\ell-5)}+\frac{3\ell-8}{2\ell-5}\\ 
& \leq \frac{27}{8}+\frac{9}{4}+\frac{3}{2}=\frac{57}{8}<\frac{60}{7}.
\end{align*}
The conclusion follows. 

For $k\geq 6$, since $n\geq 3\ell-3$, by Lemma \ref{trivialInEq} we have that
\begin{align*}
	\frac{\binom{n-2}{\ell-2}-\binom{n-5}{\ell-2}}{\binom{n-5}{\ell-2}-\binom{n-k-1}{\ell-2}}&=\frac{\binom{n-3}{\ell-3}+\binom{n-4}{\ell-3}+\binom{n-5}{\ell-3}}{\binom{n-6}{\ell-3}+\cdots +\binom{n-k-1}{\ell-3}}\leq \frac{\binom{n-3}{\ell-3}+\binom{n-4}{\ell-3}+\binom{n-5}{\ell-3}}{\binom{n-6}{\ell-3}+\binom{n-7}{\ell-3}}\\ &
	 =\frac{\frac{(n-3)(n-4)(n-5)}{(n-\ell)(n-\ell-1)(n-\ell-2)}+\frac{(n-4)(n-5)}{(n-\ell-1)(n-\ell-2)}+\frac{n-5}{n-\ell-2}}{1+\frac{n-\ell-3}{n-6}}\\ 
	& \leq \frac{\frac{(3\ell-6)(3\ell-7)(3\ell-8)}{(2\ell-3)(2\ell-4)(2\ell-5)}+\frac{(3\ell-6)(3\ell-7)}{(2\ell-4)(2\ell-5)}+\frac{3\ell-8}{2\ell-5}}{1+\frac{2\ell-6}{3\ell-9}} \\
	& \leq \frac{\frac{27}{8}+\frac{9}{4}+\frac{3}{2}}{1+\frac{2}{3}}=\frac{171}{40}<\frac{27}{4}.
\end{align*}
The conclusion follows. This completes the proof. \qed

\section{Proof of Theorem \ref{MainResult}}

Before we prove the main theorem, we state a result which is implied in the proof of \cite[Theorem 1.9]{BJX2026}.

\begin{lemma}[\cite{BJX2026}]\label{NMP1} 
Let \(\mathcal{A}\subseteq \binom{[n]}{k}\) and \(\mathcal{B}\subseteq \binom{[n]}{\ell}\) be two maximal left-compressed cross-$2$-intersecting families with $k\geq \ell\geq 3$ and $n\geq 3k-3$, $g(\mathcal{A})\in G_{*}(\mathcal{A})$, $g(\mathcal{B})\in G_{*}(\mathcal{B})$ such that $s:=\max\{s^{+}(g(\mathcal{A})),s^{+}(g(\mathcal{B}))\}$ is minimal. Let $i$ be the smallest integer such that \( g_{i}^{*}(\mathcal{A}) \neq \emptyset \). Suppose that $s\geq 5$ and $s>i\geq 3$. Define
\begin{align*}
& \mathcal{A}_{1}=\mathcal{A}\cup \mathcal{D}(g_{i}^{*}(\mathcal{A})')\subseteq \binom{[n]}{k}\ \text{and}\ \mathcal{B}_{1}=\mathcal{B}\setminus \mathcal{D}(g_{s + 2 - i}^{*}(\mathcal{B}))\subseteq \binom{[n]}{\ell}\\
& \mathcal{A}_2 = \mathcal{A} \setminus \mathcal{D}(g_{i}^{*}(\mathcal{A}))\subseteq \binom{[n]}{k}\ \text{and}\ \mathcal{B}_2 = \mathcal{B} \cup \mathcal{D}(g_{s+2-i}^{*}(\mathcal{B})')\subseteq \binom{[n]}{\ell}.
\end{align*}
Then the following statements hold.
\begin{enumerate}[(i)]
	\item If \((s,i)\not\in \{ (6, 4), (8, 5)\}\), then two inequalities $|\mathcal{A}_1||\mathcal{B}_1|\leq |\mathcal{A}||\mathcal{B}|$ and  $|\mathcal{A}_2||\mathcal{B}_2|\leq |\mathcal{A}||\mathcal{B}|$ cannot hold simultaneously.
	\item If \((s,i)\in \{ (6, 4), (8, 5)\}\), then
	$g_j(\mathcal{A})= g_j(\mathcal{B})=\emptyset \ {\rm for}\ j\leq i-1.$
\end{enumerate}	
\end{lemma}

\noindent \textbf{Proof of Theorem \ref{MainResult}}. By Fact \ref{lcl} and Lemma \ref{Sij}, we may assume that \(\mathcal{A}\subseteq\binom{[n]}{k}\) and \(\mathcal{B}\subseteq\binom{[n]}{\ell}\) are two left-compressed nontrivial cross-\(2\)-intersecting families, and we may also assume that the product \(|\mathcal{A}||\mathcal{B}|\) is maximum. 

If $\ell=2$, denote $T=\bigcup_{B\in \mathcal{B}}B$. Clearly, $|T|\geq 3$ and $T\subset A$ for any $A\in \mathcal{A}$. Hence $k\geq 3$, $|\mathcal{B}|\leq \binom{|T|}{2}$, $|\mathcal{A}|\leq \binom{n-|T|}{k-|T|}$ and  \(|\mathcal{A}||\mathcal{B}|\leq \binom{|T|}{2}\binom{n-|T|}{k-|T|}\). If $k=3$, then \(|\mathcal{A}||\mathcal{B}|\leq \binom{3}{2}\binom{n-3}{k-3}=3\). If $k>3$, since 
\begin{align*}
& \frac{\binom{x}{2}\binom{n-x}{k-x}}{\binom{x+1}{2}\binom{n-x-1}{k-x-1}}=\frac{(x-1)(n-x)}{(x+1)(k-x)}>1
\end{align*}
for $3\leq x\leq k-1$, we have \(|\mathcal{A}||\mathcal{B}|\leq \binom{3}{2}\binom{n-3}{k-3}\). The conclusion holds. 

Now we assume that $\ell\geq 3$. Let \( g(\mathcal{A}) \in G_{*}(\mathcal{A}) \) and \( g(\mathcal{B}) \in G_{*}(\mathcal{B}) \) such that \( s := \max\{s^{+}(g(\mathcal{A})), s^{+}(g(\mathcal{B}))\}\) is minimal. Since $\mathcal{H}_{n,\ell,3,2}$ and $\mathcal{I}_{n,k,3,2}$, $\mathcal{H}_{n,\ell,\ell+1,2}$ and $\mathcal{I}_{n,k,\ell+1,2}$, $\mathcal{F}_{n,k,2,1}$ and $\mathcal{F}_{n,\ell,2,1}$ are all maximal nontrivial cross-$2$-intersecting families, we have
\[
|\mathcal{A}||\mathcal{B}|  \geq \max\{h_{n,\ell,k,3,2},h_{n,\ell,k,\ell+1,2},f_{n,k,\ell,2,1}\}.
\]
If \( s = 2 \), the cross-\(2\)-intersection property of \( g(\mathcal{A}) \) and \( g(\mathcal{B}) \) implies \( g(\mathcal{A}) = g(\mathcal{B}) = \{[2]\} \). Hence, $\mathcal{A}$ and $\mathcal{B}$ are trivial cross-2-intersecting, a contradiction.
Therefore, for the remainder of the proof, we assume \( s \geq 3 \). Let \( i(2\leq i\leq k) \) be the smallest integer such that \( g_{i}^{*}(\mathcal{A}) \neq \emptyset \). Then, by (ii) of Lemma \ref{lemkl}, we have \( g_{s+2-i}^{*}(\mathcal{B}) \neq \emptyset \). 

We divide the proof into four cases based on the relationship between \(i\) and \(s\). 

{\bf Case 1}: \( i = 2 \). Then \( [s] \in g_{s+2-i}^{*}(\mathcal{B}) = g_{s}^{*}(\mathcal{B}) \). Since \( g(\mathcal{B}) \) is minimal with respect to set inclusion, we have \( g(\mathcal{B}) = \{[s]\} \). By the maximality of \( |\mathcal{A}||\mathcal{B}| \), we have  \( g(\mathcal{A}) = \binom{[s]}{2} \). Thus, 
\[
|\mathcal{A}| = \sum_{2 \leq w \leq s} \binom{s}{w} \binom{n - s}{k - w} \text{ and } |\mathcal{B}| = \binom{n - s}{\ell - s}.
\]

{\bf Subcase 1.1}: $s=3$. Then $\mathcal{A}$ and $\mathcal{B}$ are nontrivial cross-2-intersecting and \(|\mathcal{A}||\mathcal{B}|=h_{n,k,\ell,3,2}\). By Lemma \ref{hn,l,k,3,2} and the maximality of \( |\mathcal{A}||\mathcal{B}| \), we have $k=\ell$, that is, 
\begin{align*}
\mathcal{A}=\{A\in \binom{[n]}{k}\colon |A\cap [3]|\geq 2\}\ \text{and}\ \mathcal{B}=\{B\in \binom{[n]}{k}\colon [3]\subset B\}.
\end{align*} 

{\bf Subcase 1.2} $s>3$. Define the families
\[\mathcal{A}_1 = \left\{A \in \binom{[n]}{k} : |A \cap [s-1]| \geq 2\right\},\ \ \mathcal{B}_1 = \left\{B \in \binom{[n]}{\ell} : [s-1] \subset B\right\}.\]
 Clearly \(\mathcal{A}_1\) and \(\mathcal{B}_1\) are nontrivial cross-\(2\)-intersecting with
\[
|\mathcal{A}_1| = \sum_{2 \leq w \leq s-1} \binom{s-1}{w} \binom{n-s+1}{k-w} \text{ and } |\mathcal{B}_1| = \binom{n-s+1}{\ell-s+1}.
\]
Since \(\mathcal{A}_1\subset \mathcal{A}\), we have
 \begin{align*}
&|\mathcal{A}| - |\mathcal{A}_1| =(s-1)\binom{n-s}{k-2}\ \text{and}\ |\mathcal{B}_1| - |\mathcal{B}| = \binom{n-s}{\ell-s+1}.
\end{align*}
From the assumptions \(n \geq 3(k-1)\) and \(k\geq \ell\), it follows that 
\begin{align*}
s(n - \ell) - 2(n - s + 1) &= (s - 2)n - s(\ell - 2) - 2 \\
&\geq 3(s - 2)(k-1) - s(\ell - 2) - 2 \\
& = 2(s - 3)(k - 1)+s(k-\ell)+s-2 \\
& > 0.
\end{align*}
Therefore,
\begin{align*}
|\mathcal{A}_1||\mathcal{B}_1| - |\mathcal{A}||\mathcal{B}| &=\left(|\mathcal{A}|-(s-1) \binom{n-s}{k-2} \right)|\mathcal{B}_1|-|\mathcal{A}|\left(|\mathcal{B}_1|-\binom{n-s}{\ell-s+1}\right)\\
&=\binom{n-s}{\ell-s+1}\left(\sum_{2 \leq w \leq s}\binom{s}{w} \binom{n - s}{k - w}\right)-(s-1)\binom{n - s}{k - 2}\binom{n - s+1}{\ell - s+1}\\
&>\binom{n - s}{k - 2}\left(\binom{s}{2} \binom{n-s}{\ell-s+1} -(s-1)\binom{n-s+1}{\ell-s+1} \right)\\
&=(s-1)\binom{n - s}{k - 2}\binom{n-s}{\ell-s+1}  \left(\frac{s(n-l)-2(n-s+1)}{2(n-\ell)}\right)\\
&> 0. 
\end{align*}
This yields the inequality \(|\mathcal{A}_1||\mathcal{B}_1|> |\mathcal{A}||\mathcal{B}|\), which contradicts the maximality of \(|\mathcal{A}||\mathcal{B}|\).

{\bf Case 2}: \(i=s\). Since \( g(\mathcal{A}) \) is minimal, we have \( g(\mathcal{A}) = \{[s]\} \). By the maximality of \( |\mathcal{A}||\mathcal{B}| \), we have  \( g(\mathcal{B}) = \binom{[s]}{2} \). Consequently,  
\[
|\mathcal{A}| = \binom{n - s}{k - s} \text{ and } |\mathcal{B}| = \sum_{2 \leq w \leq s} \binom{s}{w} \binom{n - s}{\ell - w}.
\]

{\bf Subcase 2.1}: $s=3$. Then $\mathcal{A}$ and $\mathcal{B}$ are nontrivial cross-2-intersecting and \(|\mathcal{A}||\mathcal{B}|= h_{n,\ell,k,3,2}\), that is, 
\begin{align*}
	\mathcal{A}=\{A\in \binom{[n]}{k}\colon [3]\subset A\}\ \text{and}\ \mathcal{B}=\{B\in \binom{[n]}{\ell}\colon |B\cap [3]|\geq 2\}.
\end{align*}  

{\bf Subcase 2.2}: $s>3$.  Define the families
\[\mathcal{A}_1 =\left\{A \in \binom{[n]}{k} : [s-1] \subset A\right\},\ \mathcal{B}_1 = \left\{B \in \binom{[n]}{\ell} : |A \cap [s-1]| \geq 2 \right\}.\]
Clearly, \(\mathcal{A}_1\) and \(\mathcal{B}_1\) are nontrivial cross-\(2\)-intersecting with
\[
|\mathcal{A}_1| = \binom{n-s+1}{k-s+1} \text{ and } |\mathcal{B}_1| = \sum_{2 \leq w \leq s-1} \binom{s-1}{w} \binom{n-s+1}{\ell-w} .
\]
Following an argument analogous to that in Case 1, we obtain 
\begin{align*}
|\mathcal{A}_1||\mathcal{B}_1| - |\mathcal{A}||\mathcal{B}| &=\left(|\mathcal{B}|-(s-1) \binom{n-s}{\ell-2} \right)|\mathcal{A}_1|-|\mathcal{B}|\left(|\mathcal{A}_1|-\binom{n-s}{k-s+1}\right)\\
&>\binom{n - s}{\ell - 2}\left(\binom{s}{2} \binom{n-s}{k-s+1} -(s-1)\binom{n-s+1}{k-s+1} \right)\\
&=(s-1)\binom{n - s}{\ell - 2}\binom{n-s}{k-s+1}  \left(\frac{s(n-k)-2(n-s+1)}{2(n-k)}\right)\\
&\geq 0, 
\end{align*}
which contradicts the maximality of \(|\mathcal{A}||\mathcal{B}|\).

Therefore, for the remainder of the proof, we assume \( s-1\geq i \geq 3 \).

{\bf Case 3}: \( s = 4 \). In this case, we need only consider \( i =3 \), which implies \( g_{3}^{*}(\mathcal{A}) \neq \emptyset \) and \( g_{3}^{*}(\mathcal{B}) \neq \emptyset \). If \([2] \in g(\mathcal{A}) \cap g(\mathcal{B})\), then \( g(\mathcal{A}) = g(\mathcal{B}) = \{[2]\}\), which implies \( s = s^{+}(g(\mathcal{A})) = s^{+}(g(\mathcal{B})) = 2 \), yielding a contradiction. Therefore, \([2] \not\in g(\mathcal{A})\) or \([2] \not\in g(\mathcal{B})\).  We proceed by considering three subcases according to the membership of \([2]\) in \(g(\mathcal{A})\) and \(g(\mathcal{B})\). 

{\bf Subcase 3.1}:\ \([2] \not\in g(\mathcal{A})\) and \([2]\in g(\mathcal{B})\).  By the maximality of $|\mathcal{A}||\mathcal{B}|$, we have
\[
g(\mathcal{B}) = \{\{1,2\}, \{1,3,4\}, \{2,3,4\}\} \text{ and } g(\mathcal{A}) = \{\{1,2,3\}, \{1,2,4\} \}.
\]
Consequently, 
\[
\mathcal{B} = \left\{B \in \binom{[n]}{\ell} : [2] \subset B \text{ or } |B \cap [4]| \geq 3\right\} \text{ and}
\]
\[
\mathcal{A} = \left\{A \in \binom{[n]}{k} : [2] \subset A \text{ and } |A \cap [4]| \geq 3\right\}.
\]
It follows that 
\[|\mathcal{A}||\mathcal{B}| = h_{n,\ell,k,4,2}=\left( \binom{n-2}{k-2} - \binom{n-4}{k-2} \right)\left( \binom{n-2}{\ell-2} + 2 \binom{n-4}{\ell-3} \right) .\]
By Lemma \ref{s=4} and the maximality of $|\mathcal{A}||\mathcal{B}|$, we have $\ell=3$.

{\bf Subcase 3.2}:\ \([2] \in g(\mathcal{A})\) and \([2] \not\in g(\mathcal{B})\). By the maximality of $|\mathcal{A}||\mathcal{B}|$, we have
\[
g(\mathcal{A}) = \{\{1,2\}, \{1,3,4\}, \{2,3,4\}\} \text{ and } g(\mathcal{B}) = \{\{1,2,3\}, \{1,2,4\}\}.
\]
It follows that 
\begin{align*}
|\mathcal{A}||\mathcal{B}| = h_{n,k,\ell,4,2}= \left( \binom{n-2}{k-2}+2\binom{n-4}{k-3} \right)\left( \binom{n-2}{\ell-2}- \binom{n-4}{\ell-2} \right).
\end{align*}
By Lemmas \ref{hn,l,k,3,2} and \ref{s=4}, and the maximality of $|\mathcal{A}||\mathcal{B}|$, we have $k=\ell=3$.

{\bf Subcase 3.3}:\ \([2] \not\in g(\mathcal{A})\) and \([2] \not\in g(\mathcal{B})\).  By Lemma \ref{lemkl}(i), we have \(|F| \geq 3\) for all \(F \in g(\mathcal{A}) \cup g(\mathcal{B})\). It follows from the maximality of $|\mathcal{A}||\mathcal{B}|$ that 
\[
g(\mathcal{A}) = \binom{[4]}{3}\ \text{and}\ g(\mathcal{B}) = \binom{[4]}{3}.
\]
Consequently, 
\[
\mathcal{A} = \mathcal{F}(n, k, 2, 1) = \left\{A \in \binom{[n]}{k}\colon |A \cap [4]| \geq 3\right\} \text{ and} 
\] 
\[
\mathcal{B} = \mathcal{F}(n, \ell, 2, 1) = \left\{B \in \binom{[n]}{\ell}\colon |B \cap [4]| \geq 3\right\}.
\]  
It follows that 
\begin{align*}
	|\mathcal{A}||\mathcal{B}| =f_{n,k,\ell,2,1}= \left( 4\binom{n-4}{k-3}+\binom{n-4}{k-4} \right)\left( 4\binom{n-2}{\ell-3}+ \binom{n-4}{\ell-4} \right).
\end{align*}

{\bf Case 4}: \( s \geq 5\ (s-1\geq i\geq 3) \). Then, by (ii) of Lemma \ref{lemkl}, we have \( g_{s+2-i}^{*}(\mathcal{B}) \neq \emptyset \). Define \[\mathcal{A}_1 = \mathcal{A} \cup \mathcal{D}(g_{i}^{*}(\mathcal{A})'),\ \ \mathcal{B}_1 = \mathcal{B} \setminus \mathcal{D}(g_{s+2-i}^{*}(\mathcal{B})).\] 
By (iii) of Lemma \ref{lemkl}, the families \(\mathcal{A}_1\) and \(\mathcal{B}_1\) are cross-\(2\)-intersecting. 

If \(\mathcal{A}_1\) and \(\mathcal{B}_1\) are trivial, by Lemma \ref{transform} we have $\mathcal{A}=\mathcal{I}_{n,k,s,2}$ and $\mathcal{B}=\mathcal{H}_{n,\ell,s,2}$. It follows that $$|\mathcal{A}||\mathcal{B}|=h_{n,\ell,k,s,2}.$$
By Lemma \ref{4<s<k} and the maximality of \(|\mathcal{A}||\mathcal{B}|\), we have $s=\ell+1$.

If \(\mathcal{A}_1\) and \(\mathcal{B}_1\) are nontrivial, then
\[
|\mathcal{A}_1||\mathcal{B}_1| = \left( |\mathcal{A}| + |g_{i}^{*}(\mathcal{A})| \binom{n-s}{k-i+1} \right) \left( |\mathcal{B}| - |g_{s+2-i}^{*}(\mathcal{B})| \binom{n-s}{\ell-s-2+i} \right) \leq |\mathcal{A}||\mathcal{B}|,
\]
where the inequality follows from the maximality of \(|\mathcal{A}||\mathcal{B}|\). 

Now, define \[\mathcal{A}_2 = \mathcal{A} \setminus \mathcal{D}(g_{i}^{*}(\mathcal{A})),\ \ \mathcal{B}_2 = \mathcal{B} \cup \mathcal{D}(g_{s+2-i}^{*}(\mathcal{B})').\] 
By (iii) of Lemma \ref{lemkl}, the families \(\mathcal{A}_2\) and \(\mathcal{B}_2\) are cross-\(2\)-intersecting. 

If \(\mathcal{A}_2\) and \(\mathcal{B}_2\) are trivial, by Lemma \ref{transform} we have $\mathcal{A}=\mathcal{H}_{n,k,s,2}$ and $\mathcal{B}=\mathcal{I}_{n,\ell,s,2}$. It follows that $$|\mathcal{A}||\mathcal{B}|=h_{n,k,\ell,s,2}.$$ 
By Lemmas \ref{4<s<k} and \ref{s=4}, and the maximality of \(|\mathcal{A}||\mathcal{B}|\), we have $k=\ell$ and $s=\ell+1$.

If \(\mathcal{A}_2\) and \(\mathcal{B}_2\) are nontrivial, then
\[
|\mathcal{A}_2||\mathcal{B}_2| = \left( |\mathcal{A}| - |g_{i}^{*}(\mathcal{A})| \binom{n-s}{k-i} \right) \left( |\mathcal{B}| + |g_{s+2-i}^{*}(\mathcal{B})| \binom{n-s}{\ell-s+i-1} \right) \leq |\mathcal{A}||\mathcal{B}|. 
\]

 If \((s,i)\not\in \{ (6, 4), (8, 5)\}\), then by (i) of Lemma \ref{NMP1} we arrive at a contradiction. 
 Otherwise,  by (ii) of Lemma \ref{NMP1} we have 
\[g_j(\mathcal{A})=g_j(\mathcal{B})= \emptyset \ {\rm for}\ j\leq i-1.\]
Then, by the maximality of \(|\mathcal{A}||\mathcal{B}|\), it must hold that 
\[ g(\mathcal{A})=\binom{[s]}{i}\ {\rm and}\ g(\mathcal{B})=\binom{[s]}{i}.\]
This implies that
\[\mathcal{A}=\mathcal{F}(n,k,2,i-2)\ {\rm and}\ \mathcal{B}=\mathcal{F}(n,\ell,2,i-2).\]
Consequently, \(\mathcal{A}\) and \(\mathcal{B}\) are nontrivial \(2\)-intersecting. Since 
\begin{align*}
\frac{|\mathcal{F}(n,k,2,1)\setminus \mathcal{F}(n,k,2,2)|}{|\mathcal{F}(n,k,2,2)\setminus \mathcal{F}(n,k,2,1)|}&=\frac{\binom{4}{3}\binom{n-6}{k-3}}{\binom{4}{2}\binom{n-6}{k-4}}=\frac{2(n-k-2)}{3(k-3)}\geq \frac{2(3k-3-k-2)}{3(k-3)}>1,
\end{align*} we have $|\mathcal{F}(n,k,2,1)|>|\mathcal{F}(n,k,2,2)|$. Similar discussion shows that $|\mathcal{F}(n,\ell,2,1)|>|\mathcal{F}(n,\ell,2,2)|$,  $|\mathcal{F}(n,k,2,2)|>|\mathcal{F}(n,k,2,3)|$ and  $|\mathcal{F}(n,\ell,2,2)|>|\mathcal{F}(n,\ell,2,3)|$. It follows that $ f_{n,k,\ell,2,3}<f_{n,k,\ell,2,2}<f_{n,k,\ell,2,1}$. Hence,
$$|\mathcal{A}||\mathcal{B}|=f_{n,k,\ell,2,i-2}<f_{n,k,\ell,2,1},$$
which contradicts the maximality of \(|\mathcal{A}||\mathcal{B}|\). This completes the proof. \qed

\section{Acknowledgment}

The research work of Jingjun Bao is partially supported by the National Natural Science Foundation of China Grant No. 12471313. The research work of Lijun Ji is partially supported by the National Natural Science Foundation of China Grant No. 12271390.

\end{document}